\documentclass{amsart}
\usepackage{amsmath}
\usepackage{amssymb}
\usepackage{latexsym}
\usepackage{enumerate}
\usepackage{amsfonts}
\usepackage[all]{xy}
\usepackage{setspace}
\usepackage{cite}
\usepackage{color}

\newcommand{\ben}{\begin{enumerate}}
\newcommand{\een}{\end{enumerate}}
\newcommand{\bcen}{\begin{center}}
\newcommand{\ecen}{\end{center}}

\newtheorem{theorem}{Theorem}[section]

\newtheorem{definition}[theorem]{Definition}
\newtheorem{corollary}[theorem]{Corollary}

\newtheorem{lemma}[theorem]{Lemma}
\newtheorem{remark}[theorem]{Remark}
\numberwithin{equation}{section}
\begin{document}

\title{Cremmer-Gervais $r$-Matrices and the Cherednik Algebras of Type $GL_2$}
\author{Garrett Johnson}
\address{Department of Mathematics, University of California, Santa Barbara, CA, 93106}
\email{johnson@math.ucsb.edu}

\begin{abstract} 
\noindent  
We give an intepretation of the Cremmer-Gervais $r$-matrices for $\mathfrak{sl}_n$ in terms of actions of elements in the rational and trigonometric Cherednik algebras of type $GL_2$ on certain subspaces of their polynomial representations. This is used to compute the nilpotency index of the Jordanian $r$-matrices, thus answering a question of Gerstenhaber and Giaquinto. We also give an interpretation of the Cremmer-Gervais quantization in terms of the corresponding double affine Hecke algebra.
 \end{abstract}

\maketitle

\section{Introduction}

Let $\mathfrak{sl}_n$ denote the Lie algebra of traceless $n\times n$ matrices having entries in a field $k$. Let $V$ denote the vector representation of $\mathfrak{sl}_n$ and let $r\in\mathfrak{sl}_n\wedge\mathfrak{sl}_n\subset End (V\otimes V)$ be a skew-symmetric linear operator. Define $r_{12}:=r\otimes 1$, $r_{23}:=1\otimes r$, and $r_{13}:=P_{23}r_{12}P_{23}$ where $P_{23}$ is the permutation operator on $V^{\otimes 3}$: $P_{23}(u\otimes v\otimes w)=u\otimes w\otimes v$. An important class of operators which arise in studying Lie bialgebras and Poisson-Lie groups are those satisfying the modified classical Yang-Baxter equation (MCYBE) 
$$[r_{12},r_{13}]+[r_{12},r_{23}]+[r_{13},r_{23}]-\lambda (P_{123}-P_{213}) =0$$
for some $\lambda\in k$ (see \cite{ES}, \cite{CP} for more details). The solutions to the MCYBE are called \emph{classical $r$-matrices} and fall into two classes: those satisfying the MCYBE for $\lambda$ nonzero (resp. zero) are called \emph{quasitriangular} (resp. \emph{triangular}).

In the early 1980's, Belavin and Drinfel'd successfully classified all quasitriangular $r$-matrices in the case when $k$ is the field of complex numbers \cite{BD}. This classification gives us a solution space which we view as a disjoint union of quasi-projective subvarieties of $\mathbb{P}(\mathfrak{sl}_n\wedge\mathfrak{sl}_n)$. 
In contrast, the triangular $r$-matrices are more mysterious as there is not a constructive classification of them (except in the smaller cases of $\mathfrak{sl}_2$ and $\mathfrak{sl}_3$, see \cite{CP}, Chapter 3), only a homological interpretation exists (due to the work of Stolin \cite{St}: for details see \cite{CP}, Section 3.1.D. and \cite{ES}, Section 3.5), and there are currently few known examples.

In the paper \cite{GG}, Gerstenhaber and Giaquinto investigate the behavior along the boundaries of the aforementioned quasi-projective varieties and show that the boundary points are all triangular $r$-matrices.
In the same paper, they construct the most general class of known examples of triangular $r$-matrices, the so-called \emph{generalized Jordanian $r$-matrices} $r_{J,n}$ (see \cite{EH}, \cite{GG}). They prove that the Jordanian $r$-matrices lie on the boundary of the component corresponding to the quasitriangular Cremmer-Gervais $r$-matrices (discussed in \cite{CG}, \cite{EH}, \cite{GG}, \cite{Hodges}) and conjecture that $r_{J,n}^3=0$.

In Theorem \ref{GG Conjecture}, we prove that the nilpotency index of $r_{J,n}$ is quite different than conjectured. 
We do this by first interpreting the quantum Cremmer-Gervais $R$-matrix $R$ in terms of the double affine Hecke algebra (DAHA) ${\mathcal H}\!\!\!{\mathcal H}^{q,t}$ of type $GL_2$ (Theorem \ref{R}) and the classical Cremmer-Gervais $r$-matrix in terms of the degenerate DAHA ${\mathcal H}\!\!\!{\mathcal H}_c^\prime$ (Theorem \ref{rCG}). Using Suzuki's embedding \cite{suzuki-2005} of the rational Cherednik algebra ${\mathcal H}\!\!\!{\mathcal H}_{-c}^{\prime\prime}$ into ${\mathcal H}\!\!\!{\mathcal H}_c^\prime$, we give a simple interpretation of both the Cremmer-Gervais and Jordanian $r$-matrices as operators on the polynomial representation of ${\mathcal H\!\!\! H}_{-c}^{\prime\prime}$.
Using the relations in ${\mathcal H}\!\!\!{\mathcal H}_{-c}^{\prime\prime}$, we find that the nilpotency index of $r_{J,n}$ is $n$ when $n$ is odd, and $2n-1$ when $n$ is even. The conceptual difference between the two cases has a representation theoretic origin: the polynomial representation of the rational Cherednik algebra ${\mathcal H}\!\!\!{\mathcal H}_{c}^{\prime\prime}$ of type $A_1$
is reducible if and only if the deformation parameter $c$ has the form $n/2$ for an odd positive integer $n$ (a special case of Dunkl's theorem \cite{BEG}).\\

\text{}\\
\noindent{\bf Acknowledgements.} 
I am grateful to Milen Yakimov for many helpful discussions. The author was partially supported by NSF grant DMS-0701107.

\section{The Yang-Baxter Equations and the Double Affine Hecke Algebra}\label{quantum}
Let $k$ be a field of characteristic $0$ and let $K=k(q,t^{1/2})$.  
We begin with a 
\begin{definition} (\cite{Ch}, \cite{Ch-book}) The double affine Hecke algebra ${\mathcal H\!\!\!}{\mathcal H}^{q,t}$ of type $GL_2$ is the associative $K$-algebra with generators $X_1^{\pm 1},X_2^{\pm 1},Y_1^{\pm 1},Y_2^{\pm 1},T$ and relations
$$X_jX_j^{-1}=X_j^{-1}X_j=Y_jY_j^{-1}=Y_j^{-1}Y_j=1,$$
$$(T-t^{1/2})(T+t^{-1/2})=0,\,\,\,\,\,\,\,
TX_1T=X_2,\,\,\,\,\,\,\,
TY_2T=Y_1,$$
$$Y_2^{-1}X_1Y_2X_1^{-1}=T^2,\,\,\,\,\,\,\,
Y_1Y_2X_j = qX_jY_1Y_2,$$
$$Y_jX_1X_2=qX_1X_2Y_j,\,\,\,\,\,\,\,
[Y_1,Y_2]=0,\,\,\,\,\,\,\,
[X_1,X_2]=0,$$
for $j=1,2$.
\end{definition} The $K$-vector space $K[X_1^{\pm 1},X_2^{\pm 1}]$ can be made into a ${\mathcal H\!\!\!}{\mathcal H}^{q,t}$-module, called the $polynomial$ $representation$, defined as follows. Let $(12)$ act on $K[X_1^{\pm 1},X_2^{\pm 1}]$ by swapping variables and let $S=\frac{1-(12)}{X_1-X_2}$. For integers $a,b$ define $\Gamma_{a,b}.f(X_1,X_2):=f(q^aX_1,q^bX_2)$. The double affine Hecke algebra ${\mathcal H\!\!\! H}^{q,t}$ acts faithfully on $K[X_1^{\pm 1},X_2^{\pm 1}]$ via
$$T\mapsto t^{1/2}(12)-(t^{1/2}-t^{-1/2})X_2S,$$
$$ Y_1\mapsto T\Gamma_{0,1}(12),
\;\;\;\;\;\;\;\;\;\;Y_2\mapsto \Gamma_{0,1} (12) T^{-1},$$
and the $X$'s act via multiplication.
For an operator $R\in\text{End}_K(K[X_1^{\pm 1},X_2^{\pm 1}])$, let $R_{12}, R_{13},R_{23}$ denote the corresponding operators on $K[X_1^{\pm 1},X_2^{\pm 1},X_3^{\pm 1}]$. In \cite{GG2}, Gerstenhaber and Giaquinto introduced the \emph{modified quantum Yang-Baxter equation} (MQYBE). An operator is called a \emph{modified quantum $R$-matrix} if it satisfies the MQYBE
$$R_{12}R_{13}R_{23}-R_{23}R_{13}R_{12}
-\lambda\left(P_{123}R_{12}-P_{213}R_{23}\right)=0$$
 for some scalar $\lambda$. Here, $P_{ijk}$ denotes the permutation on the variables $X_{i}\mapsto X_j\mapsto X_k\mapsto X_i$. Furthermore, $R$ is called $unitary$ if $R(12)R(12) = 1$.

The classical analogue of the MQYBE is called the \emph{modified classical Yang-Baxter equation} (MCYBE). An operator $r$ is called a \emph{classical $r$-matrix} if it satisfies the MCYBE
$$[r_{12},r_{13}]+[r_{12},r_{23}]+[r_{13},r_{23}]-\lambda (P_{123}-P_{213}) =0$$ for some scalar $\lambda$ (for more details, see ~\cite{CP}). 
Furthermore, $r$ is called \emph{skew-symmetric} if $(12)r(12)=-r$. For shorthand, we will denote the left hand sides of the above equations by $MQYBE_\lambda (R)$ and $MCYBE_\lambda (r)$, respectively. 

\begin{theorem}\label{R} The operator $R=(12)Y_2(12)Y_2^{-1}$ is unitary and satisfies the modified quantum Yang-Baxter equation
$$R_{12}R_{13}R_{23}-R_{23}R_{13}R_{12}
=\left(1-t^{-1}\right)^2\left(P_{123}R_{12}-P_{213}R_{23}\right).$$
\end{theorem}
\begin{proof} 
It is obvious that $R$ is unitary. To show $R$ satisfies the MQYBE, we first set $v=(X_1+X_2)S$. This operator satisfies 
$MQYBE_1(v)=0$, $MCYBE_1(v)=0$, and $v^2=0$. Therefore, it follows that $MQYBE_{\lambda^2}(\exp(\lambda v))=0$ for all scalars $\lambda$.
Setting $F=\Gamma_{0,-1}$ and $\hat{R}=\exp((1-t^{-1})v)$, we have $R=F_{21}^{-1}\hat{R}F_{12}$. 
Furthermore
\begin{enumerate}[(i)]\item
$F_{12}F_{13}F_{23}=F_{23}F_{13}F_{12}$,
\item $\hat{R}_{12}F_{23}F_{13}=F_{13}F_{23}\hat{R}_{12}$,
\item  $\hat{R}_{23}F_{12}F_{13}=F_{13}F_{12}\hat{R}_{23}$.
\end{enumerate}
So $R$ is obtained by twisting $\hat{R}$ by $F$  and hence satisfies the MQYBE (c.f. \cite{Kulish-Mudrov}, Thm. 1).
\end{proof}

\section{Semiclassical Limit}\label{semiclassical}

Setting $q=e^h$ and $t=e^{ch}$, we view $K$ as a subfield of $k[[c , h]]$ and the above formulae for the $Y$'s in the polynomial representation become

$$Y_1\mapsto 1+h\left(X_1\frac{\partial}{\partial X_1}+c X_2S+\frac{c}{2}\right)+O(h^2)$$ $$Y_2\mapsto 1+h\left(X_2\frac{\partial}{\partial X_2}-cX_2S-\frac{c}{2}\right)+O(h^2)$$ Define $y_i$ as the coefficient of $h$ in $Y_i$ in the above expressions. The operators $y_1$, $y_2$ obey commutation relations which motivate the following
\begin{definition}\label{def-trig} (see Cherednik \cite{Ch-book}) The degenerate (or trigonometric) double affine Hecke algebra ${\mathcal H\!\!\!}{\mathcal H}^\prime_c$ of type $GL_2$ is the $k(c)$-algebra having generators $y_1$, $y_2$, $X_1^{\pm 1}$, $X_2^{\pm 1}$, and $(12)$ and relations
$$(12)^2=1,\,\,\,\,\,\,\,\,[X_1,X_2]=0,\,\,\,\,\,\,\,\,
\,\,\,\,\,\,\,\, (12)X_1(12)=X_2,$$
$$(12)y_1-y_2(12)=c,\,\,\,\,\,\,\,\,\,\,\,[y_1,y_2]=0,$$
$$[y_i,X_j]=\begin{cases}X_i+c (12)X_1&
\text{if }i=j\\ -c(12)X_1&\text{if }i\neq j\end{cases}$$\end{definition}

\begin{remark} \label{remark1}If $MQYBE_\lambda(1+hr+O(h^2))=0$, then $\lambda$ is of the form $\epsilon h^2+O(h^3)$ for some scalar $\epsilon$ and $MCYBE_\epsilon (r)=0$. Furthermore, if $1+hr+O(h^2)$ is unitary, then $r$ is skew-symmetric.\end{remark}

One can readily obtain the following
\begin{lemma}\label{R expanded}$R=1+h\left(y_1-y_2-c (12)\right)+O(h^2)$\end{lemma}

Using Remark \ref{remark1} together with Lemma \ref{R expanded}, we have
\begin{corollary} \label{rCG} $R$, $y_1$, $y_2$ as above
\begin{enumerate}[(i)] 
\item$r:=y_1-y_2-c (12)$ is skew-symmetric 
\item $MCYBE_{c^2}(r)=0$. 
\end{enumerate}
\end{corollary}

In this situation, $r$ is called the \emph{semiclassical limit} of  $R$. Since $r$ is homogeneous, it follows that for any natural number $n$, we can restrict the action of $r$ to the subspace $k(c)^{n,n}$ of $k(c)[X_1^\pm,X_2^\pm]$ spanned by the monomials $X_1^aX_2^b$ with $0\leq a,b\leq n-1$. Doing this yields

$$r_n=2\sum_{1\leq k<l\leq n}(k-l-c)e_{kk}\wedge e_{ll}+2c\sum_{1\leq k<l\leq n}e_{kl}\wedge e_{lk}+4c\sum_{1\leq i<k<j\leq n}e_{i+j-k,j}\wedge e_{ki}.$$

Here, $e_{ij}\wedge e_{kl}$ is the operator
$$X_1^aX_2^b\mapsto \frac{1}{2}\left(\delta_{j,a+1}\delta_{l,b+1}X_1^{i-1}X_2^{k-1}-\delta_{l,a+1}\delta_{j,b+1}X_1^{k-1}X_2^{i-1}\right).$$

\begin{remark} The formula for $r_n$ above suggests that we can view it as being in $\mathfrak{gl}_n\wedge\mathfrak{gl}_n$. Thus, we have a one-parameter family of solutions to the MCYBE over $\mathfrak{gl}_n$. Setting $c=-\frac{n}{2}$ is the only instance that $r_n$ will be in $\mathfrak{sl}_n\wedge\mathfrak{sl}_n$. 
\end{remark}

As mentioned in the introduction, the skew-symmetric solutions to the MCYBE
$$[r_{12},r_{13}]+[r_{12},r_{23}]+[r_{13},r_{23}]-\lambda (P_{123}-P_{213}) =0$$
fall into two classes: those solutions satisfying the MCYBE with $\lambda$ nonzero (resp. zero) are called \emph{quasitriangular} (resp. \emph{triangular}) \emph{$r$-matrices}.
The quasitriangular $r$-matrices (over $\mathbb{C})$  were classified in the early 1980's by Belavin and Drinfel'd using combinatorial objects on the Dynkin graph, called BD-triples \cite{BD}. 
The classification tells us that the solution space of quasitriangular $r$-matrices may be viewed as a disjoint union of quasi-projective subvarieties of $\mathbb{P}(\mathfrak{sl}_n\wedge\mathfrak{sl}_n)$: the aforementioned subvarieties are indexed by the BD-triples.
In this case $r_n$ corresponds to the maximal BD-triple obtained by deleting an extremal node. It is the so-called Cremmer-Gervais r-matrix (discussed in \cite{EH}, \cite{GG}, \cite{Hodges}, \cite{CG}.)
\begin{theorem} \label{rcg} When $c=-\frac{n}{2}$, $r_n$ is the Cremmer-Gervais $r$-matrix.
\end{theorem}
\begin{proof}
Apply the Lie algebra automorphism $e_{ij}\mapsto -e_{n+1-j,n+1-i}$ of $\mathfrak{gl}_n$ to the formula for $r_n$ above, then multiply the result by $-\frac{1}{n}$, and finally set $c = -\frac{n}{2}$; we obtain the same formula for the Cremmer-Gervais $r$-matrix $r_{CG}$ as it appears in \cite{GG}.
\end{proof}

\section{Connections with the Rational Cherednik algebra}\label{rational}
We begin this section by recalling a
\begin{definition} (see Cherednik \cite{Ch-book}) The rational Cherednik algebra (over $k(c)$) ${\mathcal H\!\!\! H}^{\prime\prime}_{c}$ of type $GL_2$ has generators $(12)$, $x_1$, $x_2$, $u_1$, $u_2$ and relations
$$(12)^2=1,\,\,\,\,\,\,\,\,\,\,\,\,\,\,
(12)x_1(12)=x_2,\,\,\,\,\,\,\,\,\,\,\,\,\,\,\,
(12)u_1(12)=u_2,$$
$$[x_1,x_2]=0,\,\,\,\,\,\,\,\,\,\,
\,\,\,\,\,\,\,\,\,\,\,\,\,\,\,\,\,\, [u_1,u_2]=0,$$
$$[u_i,x_j]=\begin{cases}1-c(12)&\text{if }i=j\\
c(12)&\text{if }i\neq j\end{cases}$$
\end{definition}

The polynomial representation $k(c)[x_1,x_2]$ of ${\mathcal H\!\!\! H}^{\prime\prime}_{c}$ is defined where
$(12)$ permutes the variables, $x_1$ and $x_2$ act via multiplication, and the $u_i$ act by the Dunkl operators
$$u_i\mapsto \frac{\partial}{\partial x_i}+c(-1)^iS.$$

In \cite{suzuki-2005}, Suzuki shows that there is an algebra embedding $\psi :{\mathcal H\!\!\! H}^{\prime\prime}_{-c}\to{\mathcal H\!\!\! H}^\prime_c$ defined on generators by
\begin{center}
\begin{tabular}{rl}
$(12)$&$\mapsto (12)$\\
$x_1$&$\mapsto X_1$\\
$x_2$&$\mapsto X_2$\\
$u_1$&$\mapsto X_1^{-1}(y_1+\frac{c}{2}-c (12))$\\
$u_2$&$\mapsto X_2^{-1}(y_2+\frac{c}{2})$\end{tabular}\end{center}

Using this algebra embedding, we see that the Cremmer-Gervais $r$-matrix has an interpretation in the rational Cherednik algebra. Here, $r_n$ corresponds to $x_1u_1-x_2u_2\in{\mathcal H\!\!\!}{\mathcal H}_{-c}^{\prime\prime}$.

In \cite{GG}, Gerstenhaber and Giaquinto provide the largest known class of examples of triangular $r$-matrices, the so-called \emph{generalized Jordanian $r$-matrices} (also discussed in \cite{EH}). They demonstrate that the Jordanian $r$-matrices lie on the boundary of the 
orbit $SL_n.r_n$, where $SL_n$ acts via the adjoint action. One can translate this into the setting of the rational Cherednik algebra. Here,
$$e^{\tau\cdot ad(u_1+u_2)}.(x_1u_1-x_2u_2) = x_1u_1-x_2u_2+\tau(u_1-u_2).$$
Therefore, we have the following
\begin{corollary}\label{rJ} \label{rj} $u_1-u_2\in{\mathcal H\!\!\!}{\mathcal H}_{-c}^{\prime\prime}$ is a boundary solution to the MCYBE (in particular, $MCYBE_0(u_1-u_2)=0$). Restricting its action to the linear subspace $k(c)^{n,n}$ and setting $c=-n/2$ corresponds to the Jordanian $r$-matrix $r_{J,n}$.
\end{corollary}

As an operator on the polynomial representation of ${\mathcal H\!\!\! H}_c^{\prime\prime}$, $$u_1-u_2 = \frac{\partial}{\partial x_1}-\frac{\partial}{\partial x_2} -2cS.$$ This gives us a one-parameter family of triangular $r$-matrices. 

In \cite{GG}, Gerstenhaber and Giaquinto conjecture that $r_{J,n}^3=0$. We use the above interpretation to show that this is not true and to compute the nilpotency index. 

\begin{theorem} \label{GG Conjecture}(Gerstenhaber-Giaquinto Conjecture) The nilpotency index of $r_{J,n}\in{\mathcal H\!\!\! H}_{n/2}^{\prime\prime}$ is $n$ when $n$ is odd and $2n-1$ when $n$ is even.\end{theorem}

\begin{proof}
Set $x=\frac{1}{2}\left(x_1-x_2\right)$, $x^\prime=
\frac{1}{2}\left(x_1+x_2\right)$. We have $[u_1-u_2,x^\prime]=0$ and
$(u_1-u_2)x^m=(m-n[[m]])x^{m-1}$. Here $[[m]]=1$ if $m$ is odd and $[[m]]=0$ otherwise. Observe that
$$x_1^ix_2^j=(-1)^jx^{i+j}+(-1)^j(i-j)x^{i+j-1}x^\prime
+\cdots.$$ One computes that in the case when $n$ is 
odd, we have
$$(u_1-u_2)^{n-1}x_2^{n-1}=(n-1)(-2)(n-3)(-4)\cdots (3-n)(2)(1-n)\neq 0$$
and for all $0\leq i,j\leq n-1$,
$(u_1-u_2)^n(x_1^ix_2^j)
=0$.
In the case when $n$ is even, one computes
$$(u_1-u_2)^{2(n-1)}(x_1^{n-1}x_2^{n-1})
=(2n-2)(n-3)(2n-4)(n-5)\cdots (3-n)(2)(1-n)\neq 0$$
and for all $1\leq i,j\leq n-1$,
$(u_1-u_2)^{2n-1}(x_1^ix_2^j)=0$.
\end{proof}

\noindent\emph{Remarks: }
For $n\geq 2$, the Cremmer-Gervais $r$-matrix $r_n$ is not nilpotent except when $c=-1/2$ and $n=2$. One can see this by
viewing $r_n$ as the operator $$X_1\frac{\partial}{\partial X_1}-X_2\frac{\partial}{\partial X_2}+c\left(X_1+X_2\right)S$$ and verifying $r_n^2(X_1-X_2)=(1+2c)(X_1-X_2)$ and 
$r_n^2(X_1^2-X_2^2)=4(1+c)(X_1^2-X_2^2)$.
So in this particular case, nilpotency is only a boundary condition.
\\\\
The conceptual difference between the even and odd cases in Theorem \ref{GG Conjecture} has a representation theoretic origin: the polynomial representation of the rational Cherednik algebra ${\mathcal H\!\!\! H}^{\prime\prime}_{c}$ of type $A_1$ is reducible if and only if the deformation parameter $c$ has the form $n/2$ for an odd positive integer $n$ (a special case of Dunkl's theorem ~\cite{BEG}).

\bibliographystyle{mrl}
\bibliography{bib}
  \end{document}